\documentclass[12pt]{article}
    \usepackage[margin=1in]{geometry}
    \usepackage[utf8]{inputenc}
    \usepackage[T1]{fontenc}    
    \usepackage{amsmath,amssymb,amsfonts,amsthm}
    \usepackage{mathtools}
    \usepackage{physics}
    \usepackage[symbol]{footmisc}
    \usepackage{url}
    \usepackage{bigints}
    \usepackage{graphicx}
    \usepackage{tabularx}
    \usepackage{xcolor}
    \usepackage{todonotes}
\newtheorem{theorem}{Theorem}
\theoremstyle{definition}

\theoremstyle{remark}
\newtheorem*{remark}{Remark}
\graphicspath{{./images/}}

\begin{document}

\begin{titlepage}
\begin{center}
    \vspace{0.5cm}
   
    \LARGE\textbf{On the correspondence between symmetries of two-dimensional autonomous dynamical systems and their phase plane realisations}
       
    \normalsize
        
    \vspace{1.0cm}
    \setcounter{footnote}{0}
    \setlength{\footnotemargin}{0.8em}
    {\large Fredrik Ohlsson\footnote{Corresponding author. E-mail: fredrik.ohlsson@umu.se}\footnote{\label{Umeå}Department of Mathematics and Mathematical Statistics, Umeå University, Sweden}, Johannes G. Borgqvist\footnote{\label{Oxford}Wolfson Centre for Mathematical Biology, Mathematical Institute, University of Oxford, United Kingdom}, and Ruth E. Baker\footnotemark[\value{footnote}]}
    
    \setlength{\footnotemargin}{1.8em}
        
    \vspace{1cm}
        
    \begin{abstract}
    We consider the relationship between symmetries of two-dimensional autonomous dynamical system in two common formulations; as a set of differential equations for the derivative of each state with respect to time, and a single differential equation in the phase plane representing the dynamics restricted to the state space of the system.
    
    Both representations can be analysed with respect to the symmetries of their governing differential equations, and we establish the correspondence between the set of infinitesimal generators of the respective formulations. Our main result is to show that every generator of a symmetry of the autonomous system induces a well-defined vector field generating a symmetry in the phase plane and, conversely, that every symmetry generator in the phase plane can be lifted to a generator of a symmetry of the original autonomous system, which is unique up to constant translations in time.
    
    The process of lifting requires the solution of a linear partial differential equation, which we refer to as the lifting condition. We discuss in detail the solution of this equation in general, and exemplify the lift of symmetries in two commonly occurring examples; a mass conserved linear model and a non-linear oscillator model. 
    \end{abstract}
    
    \vspace{0.5cm}
    
\end{center}
\end{titlepage}
   
\setcounter{footnote}{0}
\renewcommand{\thefootnote}{\arabic{footnote}}

\numberwithin{equation}{section}

\section{Introduction}
A dynamical system in two states $u$ and $v$ comprises a set of coupled first order ordinary differential equations (ODEs) for the derivatives $\dot{u}$ and $\dot{v}$ with respect to the independent variable, time $t$. Given suitable initial data, the solution of the system amounts to determining the functions $u(t)$ and $v(t)$ that describe the behaviour and evolution of the states with time. We will refer to the native space of the ODE system, parametrised by $(t,u,v)$, as the time domain. If the system is autonomous, meaning that the derivatives are not explicitly dependent on the time $t$, it is possible to eliminate time and express the system as a single equivalent ODE in the phase plane parametrised by $(u,v)$.

A powerful approach to analyse dynamical systems is \textit{linear stability analysis} \cite{arnold1973,glendinning1994}, based on the Hartman--Grobman theorem \cite{grobman1959,hartman1960}, in which the system is linearised around its steady states and the eigenvalues of the Jacobian of the linearised system are used to determine the local dynamics in the vicinity of critical points. In addition, qualitative information concerning global behaviour, e.g., asymptotic behaviour and the existence of limit cycles, can be obtained from the flow of the vector field defined by the phase plane ODE \cite{glendinning1994,strogatz2015}. For linear systems such analysis in fact provides the exact solution, and for some non-linear systems, e.g., the Lotka--Volterra model, the phase space ODE can be solved exactly, although the solution is often implicit. For general non-linear systems, however, the phase plane ODE is typically prohibitively difficult to integrate. 

A complementary approach to extract information about a dynamical system can be found in \textit{symmetry methods}, which are based on Lie group analysis of (continuous) symmetries, i.e., transformations that map solutions of the differential equations to other solutions. Symmetries can be used to find differential invariants, solve the differential equations analytically, reduce the order of models in a systematic fashion and classify models based on their symmetry properties~\cite{bluman2013symmetries,hydon2000symmetry,olver1993applications,stephani1989differential}. In the time domain, symmetries of the system of ODEs can provide information on, e.g., conservation laws of the dynamical system, while in the phase plane, symmetries can be used, e.g., to integrate the ODE or provide additional qualitative information by relating distinct solutions in the phase portrait to each other.

In order to add these approaches to the toolkit for analysis of non-linear dynamical systems, it is of both practical and conceptual interest to investigate the connection between the symmetries of the two-dimensional phase plane and three-dimensional time domain formulations. For example, symmetries of ODEs are typically difficult to obtain analytically\footnote{Even though there are always infinitely many symmetries of first order ODEs, they are often difficult to obtain as closed form solutions to the symmetry conditions.}, and the symmetry analysis may be more amenable in one representation of the dynamical system than the other. Furthermore, all models in the time domain with the same ratio $\dot{v}/\dot{u}$ share a common phase plane description and therefore also, e.g., conservation laws inferred from phase plane symmetries. The connection to symmetries in the time domain then amounts to understanding how such properties are realised in terms of the solutions to the system of ODEs.

In this paper, we present three major theorems that establish the detailed correspondence between symmetries of the time domain and the phase plane. To the best of our knowledge, this is the first time this correspondence has been considered in the literature on symmetry methods for analysis of dynamical systems. Our results offer novel insight into the way phase plane analysis and symmetry analysis are connected and can be combined to understand important aspects of the dynamical properties of non-linear autonomous two-state systems. Throughout the analysis we account for the existence of two equivalent phase plane representations, corresponding to the ambiguity in which state, $u$ or $v$, to treat as the independent variable in the reduction. Firstly, we show that symmetries in the time domain induce symmetries in the phase plane. Secondly, we demonstrate that there is no obstruction to lifting symmetries in the phase plane to symmetries in the time domain, and formulate what we refer to as the \textit{lifting condition}; a linear partial differential equation (PDE) describing how the transformation must act on time $t$ to be compatible with the action on the states $(u,v)$. By solving the lifting condition, any symmetry of the single phase plane ODE can be extended to a symmetry of a corresponding two-state system of ODEs in the time domain, which, in turn, implies that symmetry-based analysis in the phase plane can be understood in terms of temporal dynamics. Finally, we illustrate the lifting of phase plane symmetries to the time domain using two concrete autonomous examples; a linear and mass conserved system and a non-linear oscillator system.
\section{Geometric framework for symmetries}
\subsection{Symmetries in the time domain}
We consider an autonomous system of ODEs in the states $(u,v)$ given by
\begin{equation}
\label{eqn:system_ODE}
    \dot{u} = \frac{\dd u}{\dd t} = \omega_u(u,v) \,, \quad \dot{v} = \frac{\dd v}{\dd t} = \omega_v(u,v)\,.
\end{equation}
The independent and dependent variables $(t,u,v)$ parametrise the total space $M_3$ and the corresponding five-dimensional jet space $J_5$ is parameterised by the coordinate $x^{\mu}=(t,u,v,\dot{u},\dot{v})$. The subvariety in $J_5$ where Eq.~\eqref{eqn:system_ODE} is satisfied is denoted $\Delta$, and every solution $(t,u(t),v(t))$ to Eq.~\eqref{eqn:system_ODE} corresponds to a curve in $\Delta$.

We consider a 1-parameter family of Lie point transformations $\Gamma_3 : M_3 \to M_3$ generated by the vector field
\begin{equation}
\label{eqn:vector_M3}
    X = \xi(t,u,v)\partial_t + \eta_u(t,u,v)\partial_u + \eta_v(t,u,v)\partial_v \,.
\end{equation}
The transformations induce an action on $J_5$ through the prolonged generator
\begin{equation}
\label{eqn:prolonged_vector_J5}
    X^{(1)} = X + \eta_u^{(1)}(t,u,v)\partial_{\dot{u}} + \eta_v^{(1)}(t,u,v)\partial_{\dot{v}} \,,
\end{equation}
where
\begin{equation}
\label{eqn:prolonged_tangents_J5}
    \eta_u^{(1)} = D_t\eta_u - \dot{u}D_t\xi \,, \quad \eta_v^{(1)} = D_t\eta_v - \dot{v}D_t\xi \,,
\end{equation}
and
\begin{equation}
    D_t = \partial_t + \dot{u}\partial_u + \dot{v}\partial_v \,,
\end{equation}
is the total derivative in the time domain.

In order to preserve the autonomy of Eq.~\eqref{eqn:system_ODE}, which is required for the phase plane formulation described below, we require that the tangents in the dependent variables are independent of time
\begin{equation}
\label{eqn:autonomy_condition}
    \partial_t \eta_u = 0 \,, \quad \partial_t \eta_v = 0 \,.
\end{equation}

The transformation $\Gamma_3$ is a symmetry of the system in Eq.~\eqref{eqn:system_ODE}, in the sense that it preserves the space of solutions, if the generator $X^{(1)}$ satisfies the infinitesimal symmetry condition
\begin{equation}
\label{eqn:sym_con_ODE}
    \left. X^{(1)}\left( \dot{u} - \omega_u \right) \right|_{\Delta} = 0 \,, \quad \left. X^{(1)}\left( \dot{v} - \omega_v \right) \right|_{\Delta} = 0\,.
\end{equation}

\subsection{Symmetries in the phase plane}
Because the system in Eq.~\eqref{eqn:system_ODE} is autonomous the dynamics of the states can be described by reformulating Eq.~\eqref{eqn:system_ODE} as a single ODE in the $(u,v)$ phase plane
\begin{equation}
\label{eqn:system_phase}
    v' = \frac{\dd v}{\dd u} = \frac{\dot{v}}{\dot{u}} = \frac{\omega_v(u,v)}{\omega_u(u,v)} = \Omega(u,v) \,,
\end{equation}
provided that $\omega_u(u,v) \neq 0$. The independent variable $u$ and the dependent variable $v$ now parameterise a two-dimensional total space $M_2$ with the corresponding three-dimensional jet space $J_3$ being parameterised by $y^{\alpha}=(u,v,v')$. In analogy with the situation in the time domain, the subvariety in $J_3$ where Eq.~\eqref{eqn:system_phase} is satisfied is denoted $\delta$ and a solution $(u,v(u))$ is a curve in $\delta$.

A 1-parameter family of Lie point transformations $\Gamma_2 : M_2 \to M_2$ is generated by the vector field
\begin{equation}
\label{eqn:vector_M2}
    Y = \zeta_u(u,v)\partial_u + \zeta_v(u,v)\partial_v \,.
\end{equation}
As in the time domain, the transformations $\Gamma_2$ induce an action on $J_3$ through the prolonged generator
\begin{equation}
\label{eqn:prolonged_vector_J3}
    Y^{(1)} = Y + \zeta_v^{(1)}(u,v)\partial_{v'} \,,
\end{equation}
where
\begin{equation}
\label{eqn:prolonged_tangents_J3}
    \zeta_v^{(1)} = D_u\zeta_v - v'D_u\zeta_u \,,
\end{equation}
and
\begin{equation}
    D_u = \partial_u + v'\partial_v\,,
\end{equation}
is the total derivative in phase plane.

Again, in analogy with the time domain, the transformation $\Gamma_2$ is a symmetry of the system in Eq.~\eqref{eqn:system_phase} if the generator $Y^{(1)}$ satisfies the infinitesimal symmetry condition
\begin{equation}
\label{eqn:sym_con_phase}
    \left. Y^{(1)}\left( v' - \Omega \right) \right|_{\delta} = 0 \,.
\end{equation}

The equivalent formulation obtained by treating $v$ as the independent variable in the phase plane is given by
\begin{equation}
\label{eqn:system_phase_equivalent}
    u' = \frac{\dd u}{\dd v} = \frac{\dot{u}}{\dot{v}} = \frac{\omega_u(u,v)}{\omega_v(u,v)} \,,
\end{equation}
where we require $\omega_v \neq 0$, and symmetries are described by interchanging $u$ and $v$ in the constructions above. A straightforward calculation shows that the infinitesimal symmetry conditions for the two possible phase plane formulations are equivalent and the symmetries in the phase plane consequently unaffected by the choice of independent variable in the parameterisation whenever both formulations are well-defined. Geometrically, this is consistent since the equivalent phase plane representations in Eqs.~\eqref{eqn:system_phase} and \eqref{eqn:system_phase_equivalent} by definition have the same solution set, and symmetries are point transformations which preserve this set. In what follows we can therefore restrict the treatment to the phase plane description given in Eq.~\eqref{eqn:system_phase}.
\section{Reduction to the phase plane}
We now examine the reduction from the time domain to the phase plane more closely from a geometric point of view. In particular, we describe how symmetries of the two formulations are related. To this end, we consider the map $f : J_5 \to J_3$ defined by
\begin{equation}
\label{eqn:reduction_map}
    f : (t,u,v,\dot{u},\dot{v}) \mapsto \left(u,v,\frac{\dot{v}}{\dot{u}}\right) \,,
\end{equation}
or in terms of coordinates on $J_3$ $y^{\alpha}(x^{\mu}) = \left( u,v,\frac{\dot{v}}{\dot{u}} \right)$. Clearly, $f$ also maps Eq.~\eqref{eqn:system_ODE} to Eq.~\eqref{eqn:system_phase}, acts on solutions according to $f(\Delta) \subset \delta$, and restricts in a straightforward way to a map $f : M_3 \to M_2$, where we allow ourselves a slight abuse of notation by using $f$ to refer to both maps.

The push-forward of a vector field on $M_3$ by $f : M_3 \to M_2$ yields a vector field on $M_2$, but in general, since $f:J_5 \to J_3$ is not injective, the push-forward of a vector field on $J_5$ does not produce a well-defined vector field on $J_3$. However, the relation between the components of a prolonged vector field on $J_5$ implies that it is indeed pushed forward to a well-defined vector field on $J_3$. Moreover, the push-forward commutes with the prolongation.

\begin{theorem}
\label{thm:push_forward_prolongation}
The push-forward $f_*\left( X^{(1)} \right)$ of the prolonged vector field $X^{(1)}$ in Eq.~\eqref{eqn:prolonged_vector_J5} by $f: J_5 \to J_3$ in Eq.~\eqref{eqn:reduction_map} is a vector field on $J_3$, which coincides with the prolongation $\left( f_*X \right)^{(1)}$ of the push-forward of the vector field $X$ in Eq.~\eqref{eqn:vector_M3} by $f: M_3 \to M_2$.
\end{theorem}
\begin{proof}
The push-forward $f_*(X^{(1)})$ of the prolonged generator at a point $x \in J_5$ produces a vector with components
\begin{equation}
    f_*(X^{(1)})^{\alpha} = \frac{\partial y^{\alpha}}{\partial x^{\mu}} (X^{(1)})^{\mu}\,,
\end{equation}
at $y = f(x) \in J_3$. Since the non-vanishing elements of the Jacobian of $f$ are
\begin{equation}
    \frac{\partial u}{\partial u} = 1 \,, \quad \frac{\partial v}{\partial v} = 1 \,, \quad \frac{\partial v'}{\partial \dot{u}} = -\frac{\dot{v}}{\dot{u}^2} \,, \quad \frac{\partial v'}{\partial \dot{v}} = \frac{1}{\dot{u}} \,,
\end{equation}
the resulting vector is given by
\begin{equation}
    f_*(X^{(1)}) = \eta_u\partial_u + \eta_v\partial_v + \frac{1}{\dot{u}} \left( -v' \eta_u^{(1)} + \eta_v^{(1)} \right) \partial_{v'} \,.
\end{equation}
Imposing Eq.~\eqref{eqn:prolonged_tangents_J5}, which corresponds to $X^{(1)}$ being a prolonged vector field in $J_5$, produces a well-defined vector field in $J_3$
\begin{equation}
f_*(X^{(1)}) = \eta_u\partial_u + \eta_v\partial_v + \left(\partial_u\eta_v + v' \partial_v\eta_v - v' \partial_u \eta_u - (v')^2\partial_v\eta_u \right) \partial_{v'} \,,
\end{equation}
where we have enforced $\partial_t\eta_u = 0,\,\partial_t\eta_v = 0$ to ensure that the transformation preserves the autonomy of Eq.~\eqref{eqn:system_ODE}, and where we note that all dependence on the temporal tangent $\xi(t,u,v)$ cancels.

The push-forward of the generator $X$, on the other hand, is simply given by the restriction
\begin{equation}
    f_*X = \eta_u\partial_u + \eta_v\partial_v \,,
\end{equation}
and the prolongation in $(u,v)$-phase plane, according to Eq.~\eqref{eqn:prolonged_tangents_J3}, becomes
\begin{equation}
    \left( f_*X \right)^{(1)} = \eta_u\partial_u + \eta_v\partial_v + \left(\partial_u\eta_v + v' \partial_v\eta_v - v' \partial_u \eta_u - (v')^2\partial_v\eta_u \right) \partial_{v'} \,,
\end{equation}
which completes the proof.
\end{proof}

\begin{remark}
A consequence of Theorem \ref{thm:push_forward_prolongation} is that we can relax the notation and write $f_*X^{(1)}$ without ambiguity for the combined action of  push-forward and prolongation on $X$.
\end{remark}

We are now interested in determining whether symmetries in the time domain induce symmetries in phase-space, that is whether the push-forward $f_*X^{(1)}$ of the generator $X$ of a symmetry in $J_5$ generates a symmetry in $J_3$. The following theorem establishes the existence of this connection.

\begin{theorem}
If the vector field $X$ generates a symmetry of the system in Eq.~\eqref{eqn:system_ODE}, the push-forward $f_*X$ generates a symmetry of the corresponding phase space representation in Eq.~\eqref{eqn:system_phase}.
\end{theorem}
\begin{proof}
The vector field $X$ generates a symmetry if it satisfies the infinitesimal symmetry condition Eq.~\eqref{eqn:sym_con_ODE}, which in terms of components yields
\begin{equation}
\label{eqn:sym_con_ODE_u_comp}
    \left. X^{(1)}\left( \dot{u} - \omega_u \right) \right|_{\Delta} = \left( \omega_u\partial_u + \omega_v\partial_v \right) \eta_u - \left( \eta_u\partial_u + \eta_v\partial_v \right) \omega_u - \omega_u \left. \left( D_t\xi \right) \right|_{\Delta} = 0 \,,\!
\end{equation}
and
\begin{equation}
\label{eqn:sym_con_ODE_v_comp}
    \left. X^{(1)}\left( \dot{v} - \omega_v \right) \right|_{\Delta} = \left( \omega_u\partial_u + \omega_v\partial_v \right) \eta_v - \left( \eta_u\partial_u + \eta_v\partial_v \right) \omega_v - \omega_v \left. \left( D_t\xi \right) \right|_{\Delta} = 0 \,,
\end{equation}
where once again we have enforced $\partial_t\eta_u = 0$ and $\partial_t\eta_v = 0$ to preserve autonomy. The corresponding infinitesimal symmetry condition Eq.~\eqref{eqn:sym_con_phase} in phase plane, in terms of components, is given by 
\begin{multline}
\label{eqn:sym_con_phase_comp}
    \left. f_*X^{(1)}\left( v' - \Omega \right) \right|_{\delta} = \frac{1}{\omega_u^2} \left[ \omega_u \left( \vphantom{\frac{1}{\omega_u^2}} \left( \omega_u\partial_u + \omega_v\partial_v \right) \eta_v - \left( \eta_u\partial_u + \eta_v\partial_v \right) \omega_v \right) \right. \\ \left. - \omega_v \left( \vphantom{\frac{1}{\omega_u^2}} \left( \omega_u\partial_u + \omega_v\partial_v \right) \eta_u - \left( \eta_u\partial_u + \eta_v\partial_v \right) \omega_u \right) \right] = 0 \,.
\end{multline}
The vanishing of the quantity $\omega_u \left. X^{(1)}\left( \dot{v} - \omega_v \right) \right|_{\Delta} - \omega_v \left. X^{(1)}\left( \dot{u} - \omega_u \right) \right|_{\Delta}$, where terms containing $D_t \xi$ cancel, follows from Eqs.~\eqref{eqn:sym_con_ODE_u_comp} and \eqref{eqn:sym_con_ODE_v_comp}, and implies Eq.~\eqref{eqn:sym_con_phase_comp} which establishes the theorem. 
\end{proof}

\begin{remark}
Clearly, the reduction to the phase plane is independent of the choice of independent variable $u$ or $v$ in the parameterisation of $(u,v)$-space. In particular, the symmetry generator $X = \xi\partial_t + \eta_u\partial_u + \eta_v\partial_v$ reduces to the same vector field $\eta_u\partial_u + \eta_v\partial_v$ in both formulations. All constructions and proofs in this section are therefore unaffected by interchanging the roles $u$ and $v$, and $X$ consequently induces the same symmetry generator in both equivalent phase plane formulations.
\end{remark}

\section{Lifting to the time domain}
In the previous section we established that a symmetry generator $X$ in the time domain induces a symmetry generator $f_*X$ in phase space. We now investigate the converse situation and consider lifting a generator $Y$ of a symmetry from phase space to the time domain, to determine whether all symmetries of the phase space formulation Eq.~\eqref{eqn:system_phase} can be obtained by reduction from symmetries of Eq.~\eqref{eqn:system_ODE} in the time domain.

\subsection{Lifting the generator}
Lifting the vector field $Y$ in Eq.~\eqref{eqn:vector_M2} amounts to introducing a smooth tangent in the time direction according to
\begin{equation}
\label{eqn:lift_vector_M3}
    \hat{Y} = \xi(t,u,v)\partial_t + Y\,.
\end{equation}
The lift $\hat{Y}$ is not unique, but by construction satisfies $f_*\hat{Y} = Y$ and $f_*\hat{Y}^{(1)} = Y^{(1)}$ according to the results of the previous section. Furthermore, if $Y$ generates a symmetry in $J_3$ it can be lifted to a symmetry generator in $J_5$.

\begin{theorem}
If the vector field $Y$ generates a symmetry of the phase space representation Eq.~\eqref{eqn:system_phase}, the lift $\hat{Y}$ generates a symmetry of the system in Eq.~\eqref{eqn:system_ODE} if and only if it satisfies
\begin{equation}
\label{eqn:lifting_condition}
    \left. \left( D_t \xi \right) \right|_{\Delta} =  \frac{1}{\omega_u} \left( \vphantom{\frac{1}{\omega_u}} \left( \omega_u\partial_u + \omega_v\partial_v \right) \zeta_u - \left( \zeta_u\partial_u + \zeta_v\partial_v \right) \omega_u \right) \,.
\end{equation}
\end{theorem}
\begin{proof}
The vector field $Y$ generates a symmetry if it satisfies the infinitesimal symmetry condition Eq.~\eqref{eqn:sym_con_phase} which was shown above is equivalent to 
\begin{multline}
\label{eqn:sym_con_phase_comp_lift}
    \omega_u \left( \left( \omega_u\partial_u + \omega_v\partial_v \right) \zeta_v - \left( \zeta_u\partial_u + \zeta_v\partial_v \right) \omega_v \right) = \omega_v \left( \left( \omega_u\partial_u + \omega_v\partial_v \right) \zeta_u - \left( \zeta_u\partial_u + \zeta_v\partial_v \right) \omega_u \right) \,.
\end{multline}
Similarly, the lift $\hat{Y}$ generates a symmetry in the time domain if it satisfies the infinitesimal symmetry condition, which amounts to
\begin{equation}
\label{eqn:sym_con_ODE_u_comp_lift}
     \omega_u \left. \left( D_t\xi \right) \right|_{\Delta} = \left( \omega_u\partial_u + \omega_v\partial_v \right) \zeta_u - \left( \zeta_u\partial_u + \zeta_v\partial_v \right) \omega_u \,,
\end{equation}
and
\begin{equation}
\label{eqn:sym_con_ODE_v_comp_lift}
    \omega_v \left. \left( D_t\xi \right) \right|_{\Delta} = \left( \omega_u\partial_u + \omega_v\partial_v \right) \zeta_v - \left( \zeta_u\partial_u + \zeta_v\partial_v \right) \omega_v \,.
\end{equation}

In order to establish the theorem, we must show that there is no obstruction to simultaneously satisfying Eq.~\eqref{eqn:sym_con_ODE_u_comp_lift} and Eq.~\eqref{eqn:sym_con_ODE_v_comp_lift}, and determine the condition that $\xi(t,u,v)$ is required to fulfil to make $\hat{Y}$ a symmetry generator.

We consider first the case $\omega_v = 0$. Then Eq.~\eqref{eqn:sym_con_phase_comp_lift} reduces to $\partial_u\zeta_v = 0$ and implies that Eq.~\eqref{eqn:sym_con_ODE_v_comp_lift} is identically satisfied. Consequently, there is no obstruction and Eq.~\eqref{eqn:sym_con_ODE_u_comp_lift} reduces to the lifting condition
\begin{equation}
    \left. \left( D_t \xi \right) \right|_{\Delta} =  \frac{1}{\omega_u} \left( \vphantom{\frac{1}{\omega_u}}  \omega_u\partial_u\zeta_u - \left( \zeta_u\partial_u + \zeta_v\partial_v \right) \omega_u \right) \,.
\end{equation}

In the case $\omega_v \neq 0$, we observe that Eq.~\eqref{eqn:sym_con_phase_comp_lift} immediately implies that the lifting conditions
\begin{equation}
    \left. \left( D_t \xi \right) \right|_{\Delta} =  \frac{1}{\omega_u} \left( \vphantom{\frac{1}{\omega_u}} \left( \omega_u\partial_u + \omega_v\partial_v \right) \zeta_u - \left( \zeta_u\partial_u + \zeta_v\partial_v \right) \omega_u \right) \,,
\end{equation}
and
\begin{equation}
    \left. \left( D_t \xi \right) \right|_{\Delta} =  \frac{1}{\omega_v} \left( \vphantom{\frac{1}{\omega_v}} \left( \omega_u\partial_u + \omega_v\partial_v \right) \zeta_v - \left( \zeta_u\partial_u + \zeta_v\partial_v \right) \omega_v \right) \,,
\end{equation}
produced by Eqs.~\eqref{eqn:sym_con_ODE_u_comp_lift} and \eqref{eqn:sym_con_ODE_v_comp_lift}, respectively, are equivalent and that there is consequently no obstruction.
\end{proof}

\begin{remark}
We note that, in general, the lifted symmetry generator $\hat{Y}$ will not be fibre-preserving since $\xi$ can contain a dependence on the dependent variables $u$ and $v$.
\end{remark}

\subsection{Solving the lifting condition}
The general form of the lifting condition Eq.~\eqref{eqn:lifting_condition} for the system in Eq.~\eqref{eqn:system_ODE} is given by
\begin{equation}
\label{eqn:lifting_condition_general}
    \partial_t \xi + \omega_u\partial_u\xi + \omega_v\partial_v\xi = G(u,v) \,,
\end{equation}
where the right-hand side $G(u,v)$ depends on the states through $\omega_u$, $\omega_v$, $\zeta_u$ and $\zeta_v$, but contains no explicit dependence on $t$ due to autonomy. Equation \eqref{eqn:lifting_condition_general} is a linear PDE for the time tangent $\xi(t,u,v)$ which can, in general, be solved using the method of characteristics.

The corresponding characteristic system is
\begin{equation}
\label{eqn:characteristic_system_lifting}
    \frac{\dd t}{\dd s} = 1 \,, \quad \frac{\dd u}{\dd s} = \omega_u \,, \quad \frac{\dd v}{\dd s} = \omega_v \,, \quad \frac{\dd \xi}{\dd s} = G(u,v) \,,
\end{equation}
where $s$ parameterises the characteristic curves which are simply the solutions to the original system in Eq.~\eqref{eqn:system_ODE}, a result which can also be deduced from the fact that the total derivative $D_t$ is the generator of time evolution for the system. The remaining ODE for $\xi$ restricted to a characteristic curve can then be solved to produce
\begin{equation}
\label{eqn:characteristic_system_solution}
    \xi = \int G(u(s),v(s)) \dd s + F \,,
\end{equation}
where $F$ is constant on each characteristic. As a consequence, $F$ is an arbitrary function of the constants of motion of the system in Eq.~\eqref{eqn:system_ODE} and the effect of a non-vanishing function $F$ is a constant shift in time of the entire solution curve. In other words, the general lifted symmetry generator
\begin{equation}
\label{eqn:lifting_condition_general_solution}
    \hat{Y} = \left[ \left(\int G(u(s),v(s)) \dd s\right)\partial_t + \zeta_u\partial_u + \zeta_v\partial_v \right] + F\partial_t \,,
\end{equation}
can be decomposed into one symmetry transformation depending on the non-homogeneous part of the solution Eq.~\eqref{eqn:characteristic_system_solution} and one translation $F \partial_t$ in time. In fact, time translation is a manifest symmetry corresponding to the autonomy of the system in Eq.~\eqref{eqn:system_ODE} meaning that Eq.~\eqref{eqn:lifting_condition_general_solution} can be interpreted as a linear combination of two different generators of the full symmetry group of the system in Eq.~\eqref{eqn:system_ODE}.

\begin{remark}
Just as for the reduction to the phase plane, the lifting to the time domain is independent of which of the formulations in Eqs.~\eqref{eqn:system_phase} or \eqref{eqn:system_phase_equivalent} is used. In particular, the infinitesimal symmetry condition is the same in the two formulations meaning that the two lifting conditions in Eqs.~\eqref{eqn:sym_con_ODE_u_comp_lift} and \eqref{eqn:sym_con_ODE_v_comp_lift} are equivalent and produce the same general solution given in Eq.~\eqref{eqn:characteristic_system_solution}.
\end{remark}

\section{Examples of lifting phase plane symmetries}
Having established the theory regarding lifting phase plane symmetries to the time domain, we will consider two concrete example models; a linear model describing a mass-conserved system and a non-linear oscillator model. Although we established in the previous section that there is no obstruction to lifting symmetries, it is still a non-trivial process to solve the lifting condition given in Eq.~\eqref{eqn:lifting_condition}. However, the condition on the tangent $\xi(t,u,v)$ in the time direction is by construction a linear PDE making it amenable to the method of characteristics as we demonstrated in the previous section. Consequently, it is always possible to solve the lifting condition, although in general the solution will be implicit. 

\subsection{Mass-conserved linear model}
As our first example, we consider the linear model
\begin{equation}
\label{eqn:linear_model_ODE}
    \frac{\dd u}{\dd t} = -u+v \,, \quad \frac{\dd v}{\dd t} = u-v \,,
\end{equation}
which reduced to the phase plane is described by the single ODE
\begin{equation}
\label{eqn:linear_model_phase}
    \dv{v}{u}=-1 \,.
\end{equation}
The general solution to the system is given by the phase plane trajectories
\begin{equation}
\label{eqn:linear_model_solution}
    u+v=C \,, \quad C\in\mathbb{R} \,,
\end{equation}
which implies that the total mass is a constant of motion for the dynamics of the system. Solutions to the model in the phase plane and the time domain are illustrated in Fig.~\ref{fig:linear_model_solutions}.

\begin{figure}[ht!]
\begin{center}
\includegraphics[width=\textwidth]{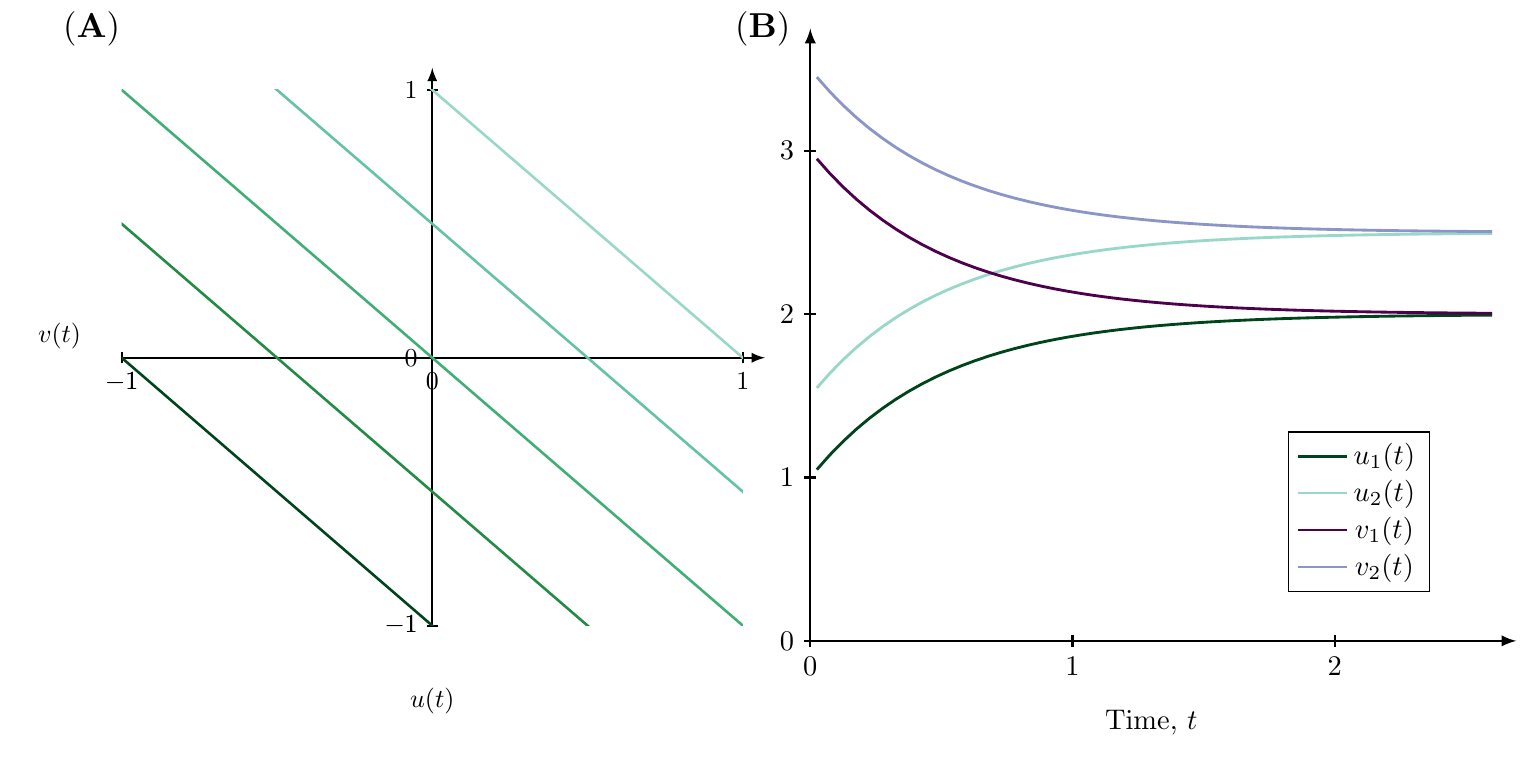}
\caption{\textit{The dynamics of the mass-conserved linear model}. Multiple solutions of the mass-conserved linear model are illustrated in (\textbf{A}) the $(u,v)$ phase plane and (\textbf{B}) the time domain.}
\label{fig:linear_model_solutions}
\end{center}
\end{figure}

The model in Eq.~\eqref{eqn:linear_model_phase} has two symmetries generated by the phase plane vector fields
\begin{equation}
\label{eqn:linear_model_scaling_Y}
    Y_S = u\partial_u+v\partial_v \,,
\end{equation}
and
\begin{equation}
\label{eqn:linear_model_genrot_Y}
    Y_G = \left(\dfrac{u+v}{u-v}\right)\left(v\partial_u-u\partial_v\right) \,.
\end{equation}
The vector field $Y_S$ in Eq.~\eqref{eqn:linear_model_scaling_Y} generates a uniform scaling and $Y_G$ in Eq.~\eqref{eqn:linear_model_genrot_Y} generates a generalised rotation. The action of the the two corresponding symmetry transformations $\Gamma_2^S$ and $\Gamma_2^G$ on the phase plane is illustrated in Fig.~\ref{fig:linear_model_symmetries_phase}.

\begin{figure}[ht!]
\begin{center}
\includegraphics[width=\textwidth]{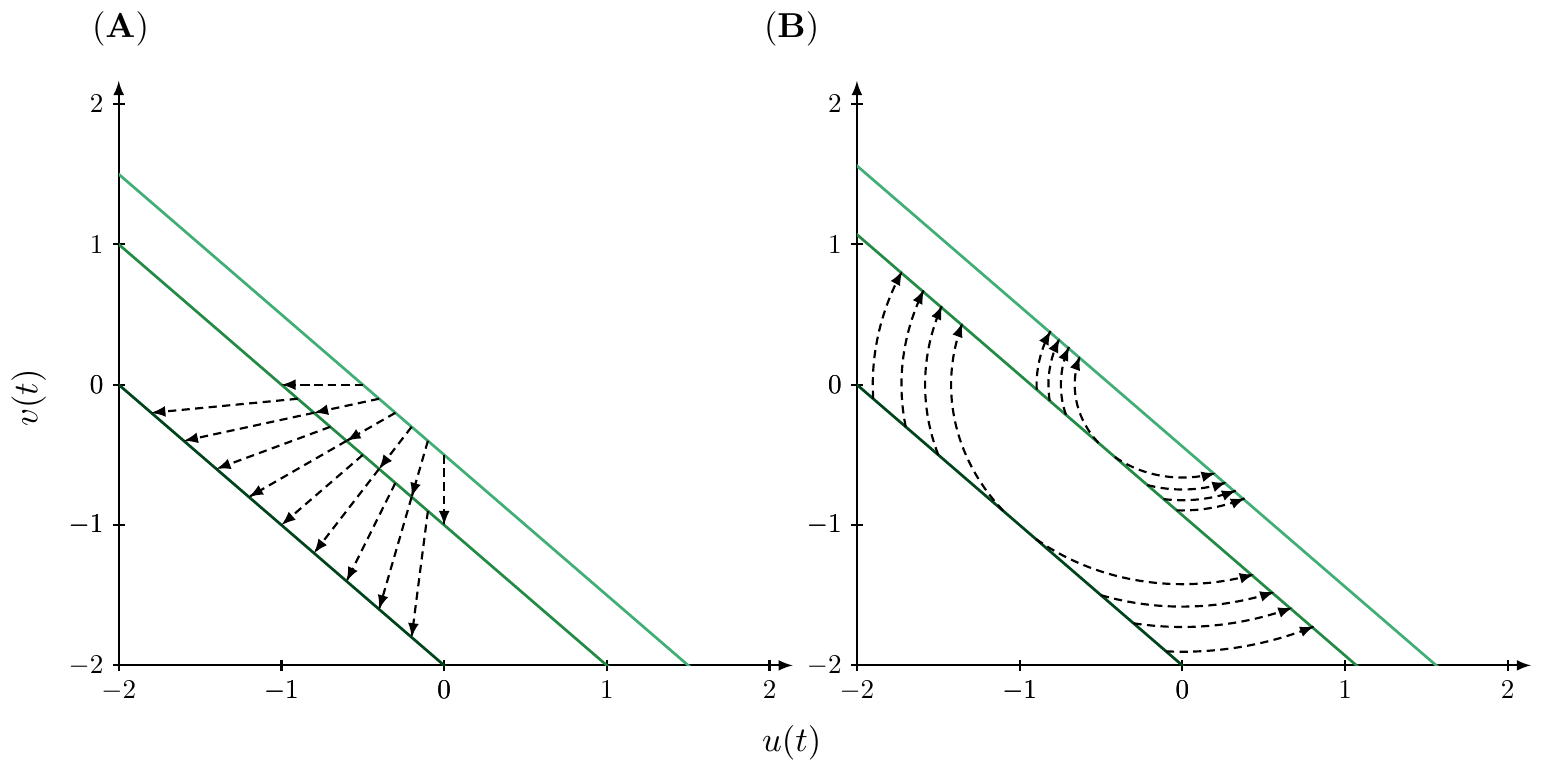}
\caption{\textit{Action of the phase plane symmetries for the mass conserved model}. Dashed arrows represent (\textbf{A}) the scaling symmetry $\Gamma_2^S$ generated by $Y_S$ and (\textbf{B}) the generalised rotation symmetry $\Gamma_2^G$ generated by $Y_G$.}
\label{fig:linear_model_symmetries_phase}
\end{center}
\end{figure}

We will now proceed to lift the generators to the time domain by solving Eq.~\eqref{eqn:lifting_condition}. Starting with the scaling symmetry generator $Y_S$ in Eq.~\eqref{eqn:linear_model_scaling_Y}, the lifting condition is
\begin{equation}
\label{eqn:lifting_cond_scaling}
    \partial_t\xi_S+(v-u)\partial_u\xi_S+(u-v)\partial_v\xi_S=0 \,.
\end{equation}
Using the method of characteristics, we obtain the following family of solutions for the infinitesimal $\xi_S(t,u,v)$
\begin{equation}
\label{eqn:xi_scaling}
    \xi_S(t,u,v) = F(u+v) \,, \quad F\in\mathcal{C}^1(\mathbb{R}) \,,
\end{equation}
parametrised by the choice of an arbitrary smooth function $F$. Consequently, the lift of Eq.~\eqref{eqn:linear_model_scaling_Y} is the family of generators of time domain symmetries given by
\begin{equation}
\label{eqn:linear_model_scaling_X}
    \hat{Y}_S = F(u+v)\partial_t+u\partial_u+v\partial_v \,, \quad F\in\mathcal{C}^1(\mathbb{R}).
\end{equation}

For the generator given in Eq.~\eqref{eqn:linear_model_genrot_Y}, we proceed in the same way to solve the lifting condition, which takes the form
\begin{equation}
\label{eqn:lifting_cond_genrot}
    \partial_t\xi_G+(v-u)\partial_u\xi_G+(u-v)\partial_v\xi_G = -2\left(\frac{u+v}{u-v}\right)^2 \,,
\end{equation}
for the time infinitesimal $\xi_G(t,u,v)$ using the method of characteristics. The solution is \begin{equation}
\label{eqn:xi_genrot}
    \xi_G{\left(t,u,v \right)} = -\frac{1}{2}\left(\frac{u+v}{u-v}\right)^2 +F{\left(u + v \right)}\,, \quad F\in\mathcal{C}^1(\mathbb{R}) \,,
\end{equation}
and the corresponding lift of the generator in Eq.~\eqref{eqn:linear_model_genrot_Y} to the time domain is the family of generators given by
\begin{equation}
\label{eqn:linear_model_genrot_X}
    \hat{Y}_G = \left[ -\frac{1}{2}\left(\frac{u+v}{u-v}\right)^2 +F{\left(u+v \right)}\right]\partial_t + \left(\frac{u+v}{u-v}\right)v\partial_u - \left(\frac{u+v}{u-v}\right)u\partial_v \,, \quad F\in\mathcal{C}^1(\mathbb{R}) \,.
\end{equation}

The transformations $\Gamma_3^S$ and $\Gamma_3^G$ in the time domain, generated by $\hat{Y}_S$ and $\hat{Y}_G$, respectively, are obtained as the integral curves using the exponential map $\exp\mathrm{(}\varepsilon \hat{Y}\mathrm{)}$, where $\varepsilon$ parameterises the integral curve, or equivalently the corresponding 1-dimensional Lie group of transformations. The action on solutions to the system in Eq.~\eqref{eqn:linear_model_ODE} of the transformations is illustrated in Fig.~\ref{fig:linear_model_symmetries_ODE} for different choices of the arbitrary function $F$ in Eqs.~\eqref{eqn:linear_model_scaling_X} and \eqref{eqn:linear_model_genrot_X}.

\begin{figure}[ht!]
\begin{center}
\includegraphics[width=\textwidth]{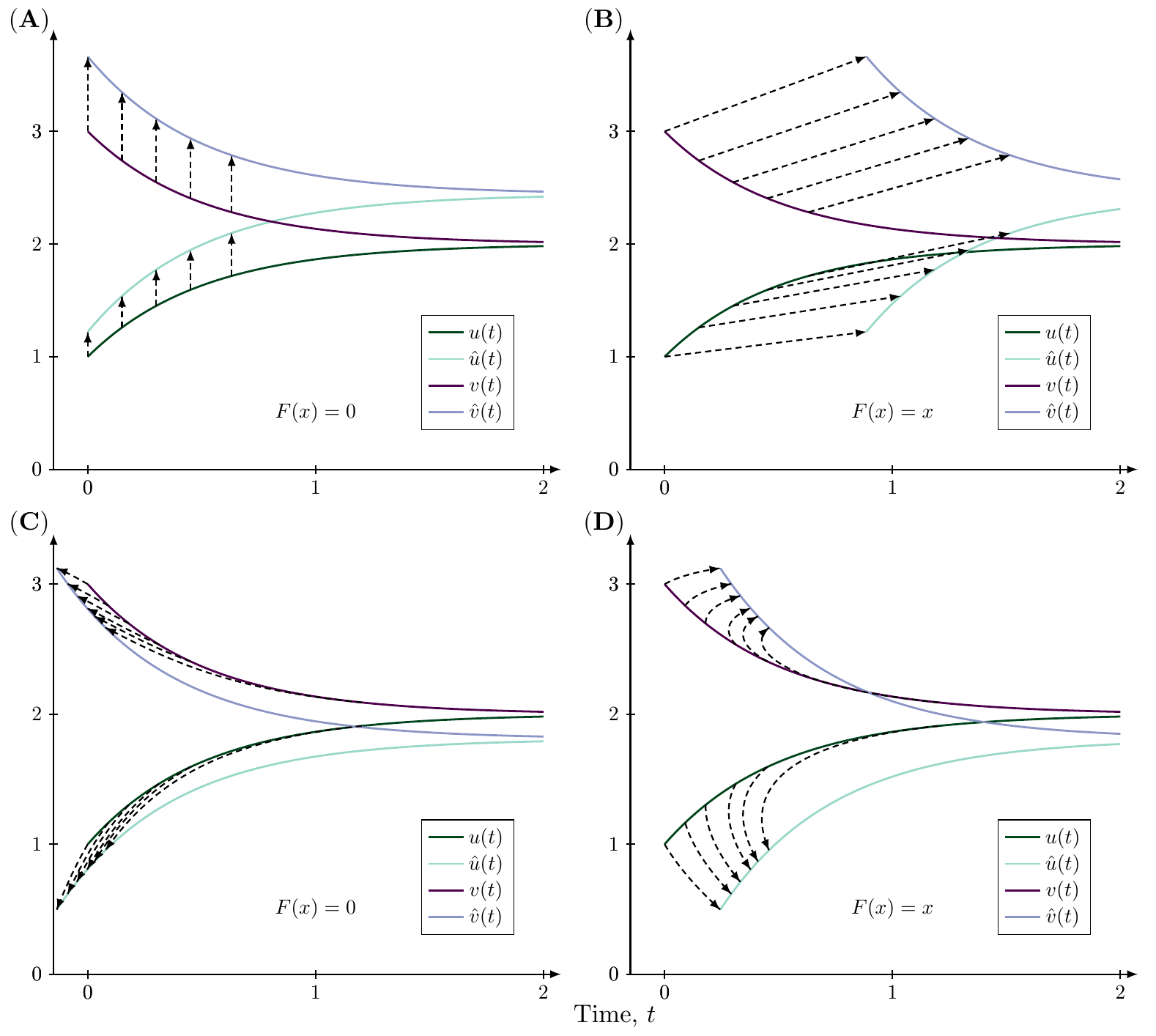}
\caption{\textit{Action of lifted phase plane symmetries for the mass-conserved linear model.} Top row: Original solution curves $\left(u(t),v(t)\right)$ and transformed solution curves $\left( \hat{u}(t),\hat{v}(t) \right) = \Gamma_3^S\left(u(t),v(t)\right)$ generated by $\hat{Y}_S$ in Eq.~\eqref{eqn:linear_model_scaling_X} for $F(x)=0$ (\textbf{A}) and $F(x)=x$ (\textbf{B}). Bottom row: Original solution curves $\left(u(t),v(t)\right)$ and transformed solution curves $\left( \hat{u}(t),\hat{v}(t) \right) = \Gamma_3^G\left(u(t),v(t)\right)$ generated by $\hat{Y}_G$ in Eq.~\eqref{eqn:linear_model_genrot_X} for $F(x)=0$ (\textbf{C}) and $F(x)=x$ (\textbf{D}). Dashed arrows represent the transformations $\Gamma_3^S$ and $\Gamma_3^G$, respectively.} 
\label{fig:linear_model_symmetries_ODE}
\end{center}
\end{figure}

\subsection{A non-linear oscillator model}
Next, we consider the following non-linear system of ODEs in the time domain
\begin{equation}
\label{eqn:nonlinear_model_ODE}
    \frac{\dd u}{\dd t} = u-v-u^3-uv^2 \,, \quad \frac{\dd v}{\dd t} = u+v-v^3-u^2v \,.
\end{equation}
The corresponding phase plane ODE is given by
\begin{equation}
\label{eqn:nonlinear_model_phase}
    \frac{\dd v}{\dd u}=\frac{u+v-v^3-u^2v}{u-v-u^3-uv^2} \,,
\end{equation}
which has a rotation symmetry generated by the phase plane vector field~\cite{hydon2000symmetry}
\begin{equation}
\label{eqn:nonlinear_model_rotation_Y}
    Y_R=-v\partial_u+u\partial_v \,.
\end{equation}
The dynamics of the model in Eq.~\eqref{eqn:nonlinear_model_ODE} is illustrated in Fig.~\ref{fig:nonlinear_model_solutions} and the action on the phase plane generated by $Y_R$ is illustrated in Fig.~\ref{fig:nonlinear_model_symmetries_phase}.

\begin{figure}[ht!]
\begin{center}
\includegraphics[width=\textwidth]{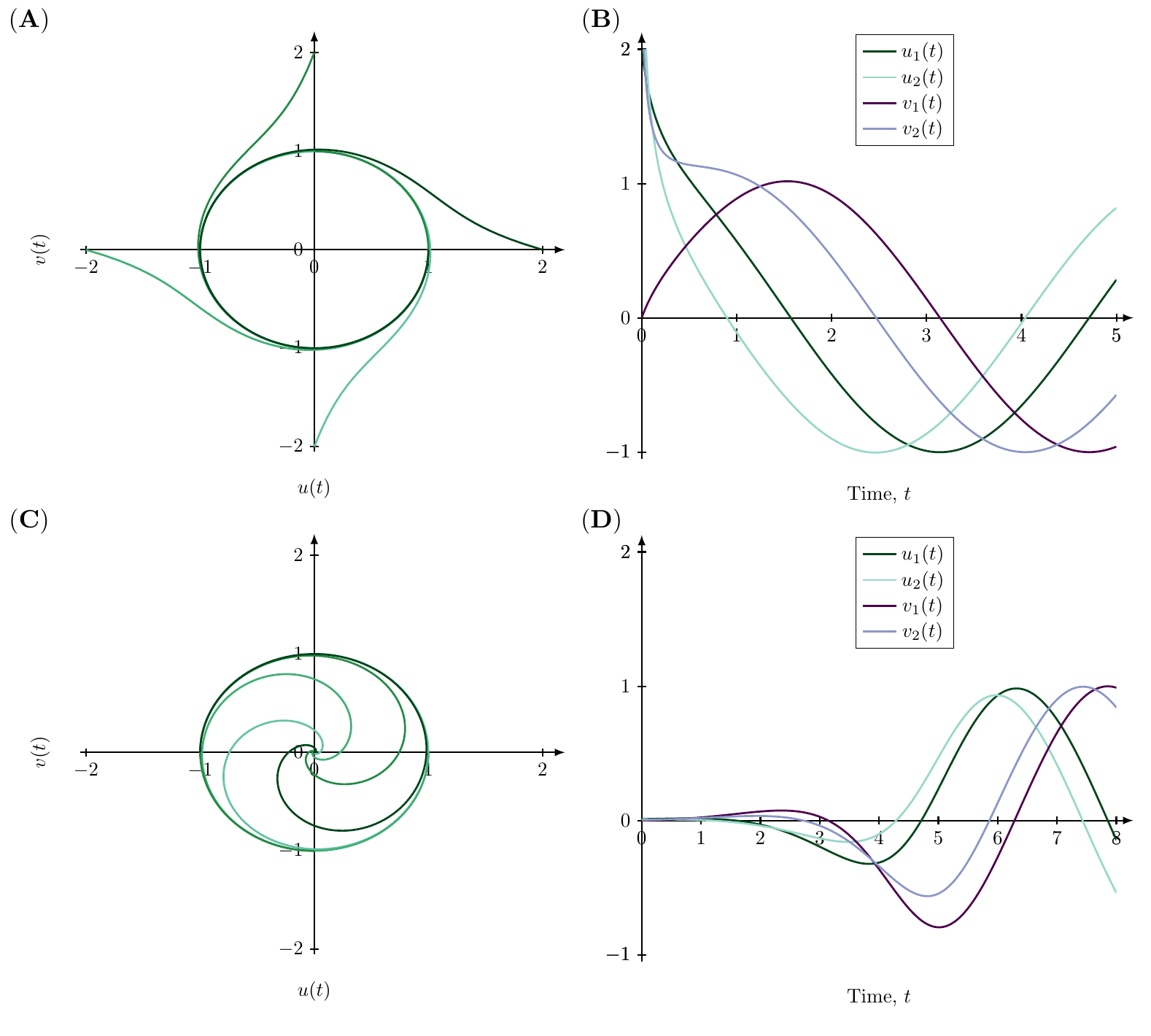}
\caption{\textit{The dynamics of the non-linear oscillator model}. Multiple solutions of the non-linear oscillator model are illustrated for $r>1$ in (\textbf{A}) the $(u,v)$ phase plane and (\textbf{B}) the time domain, and for $r<1$ in (\textbf{C}) the $(u,v)$ phase plane and (\textbf{D}) the time domain.}
\label{fig:nonlinear_model_solutions}
\end{center}
\end{figure}

\begin{figure}[ht!]
\begin{center}
\includegraphics[width=\textwidth]{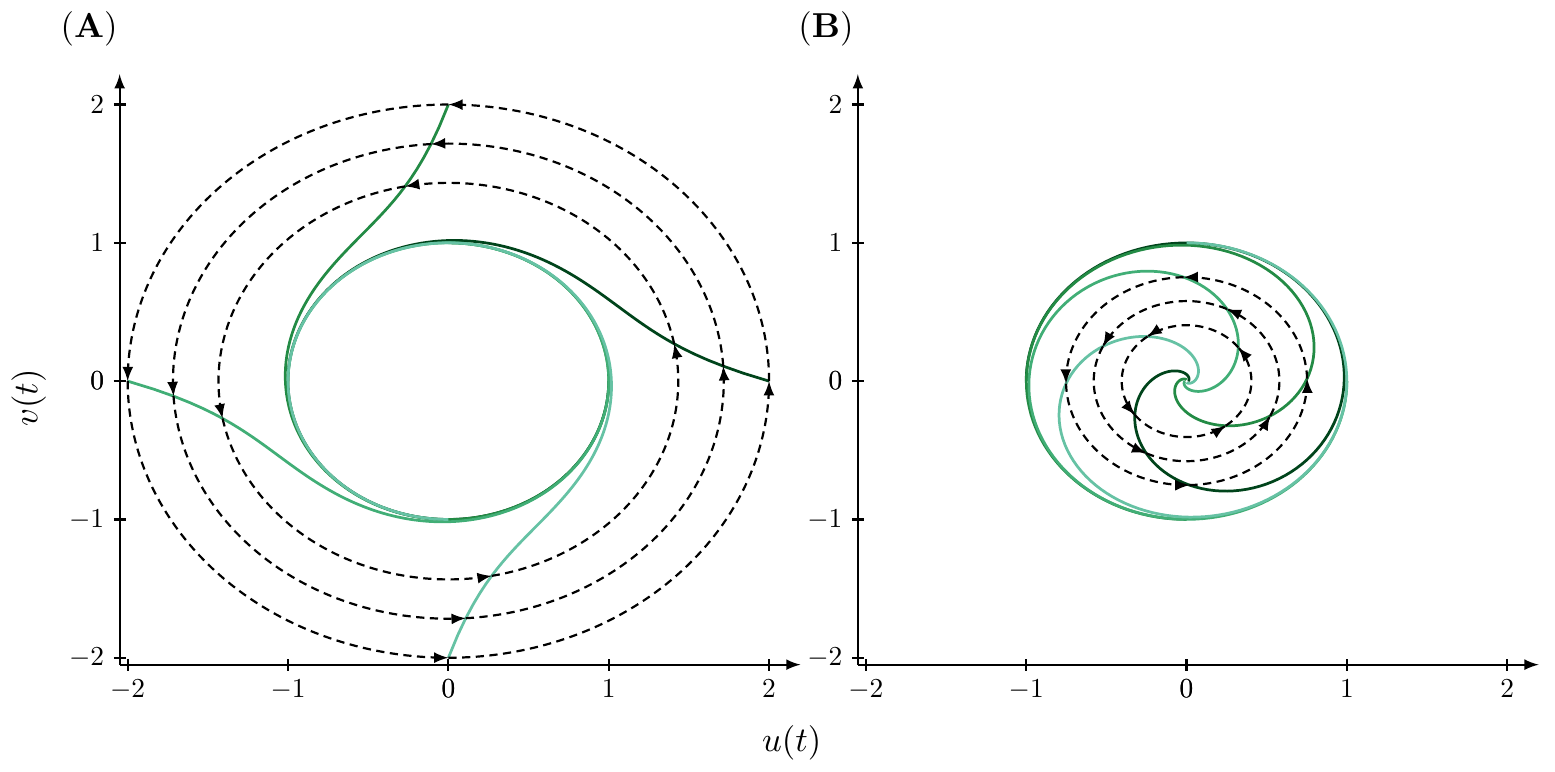}
\caption{\textit{Action of the phase plane symmetry for the non-linear oscillator model}. Dashed arrows represent the rotation symmetry $\Gamma_2^R$ acting on oscillating trajectories for (\textbf{A}) $r>1$ and (\textbf{B}) $r<1$.}
\label{fig:nonlinear_model_symmetries_phase}
\end{center}
\end{figure}

The lift of the phase plane symmetry generator $Y_R$ in Eq.~\eqref{eqn:nonlinear_model_rotation_Y}, is facilitated by a transition to polar coordinates $(r,\theta)$ defined by
\begin{equation}
\label{eqn:polar_coord}
    u=r\cos(\theta)\,, \quad v=r\sin(\theta)\,.
\end{equation}
In these coordinates, the system in the time domain takes the form
\begin{equation}
\label{eqn:nonlinear_model_ODE_polar}
    \frac{\dd \theta}{\dd t} = 1 \,, \quad \frac{\dd r}{\dd t} = r(1-r^2) \,,
\end{equation}
the corresponding phase plane ODE is
\begin{equation}
\label{eqn:nonlinear_model_phase_polar}
    \frac{\dd r}{\dd \theta}=r(1-r^2) \,,
\end{equation}
and the symmetry generator is simply
\begin{equation}
\label{eqn:nonlinear_model_rotation_Y_polar}
    Y_R=\partial_{\theta} \,.
\end{equation}

The lifting condition for the phase plane symmetry generator $Y_R$ in Eq.~\eqref{eqn:nonlinear_model_rotation_Y_polar} and time domain system in Eq.~\eqref{eqn:nonlinear_model_ODE_polar} is given by
\begin{equation}
\label{eqn:lifting_cond_rotation}
    \partial_t\xi_R + \partial_{\theta}\xi_R + r \left(1-r^2\right) \partial_r \xi_R = 0 \,,
\end{equation}
and solving for the unknown time infinitesimal $\xi_R(t,r,\theta)$ using the method of characteristics we obtain 
\begin{equation}
\label{eqn:xi_rotation_polar}
    \xi_R(t,r,\theta) = F\left(\ln\left(\frac{r}{\sqrt{|1-r^2|}}\right)-\theta\right) \,, \quad F\in\mathcal{C}^1(\mathbb{R}) \,.
\end{equation}
Consequently, the lifted infinitesimal generator in the time domain is given by
\begin{equation}
\label{eqn:nonlinear_model_rotation_X_polar}
    \hat{Y}_R = F\left(\ln\left(\frac{r}{\sqrt{|1-r^2|}}\right)-\theta\right)\partial_t + \partial_\theta \,, \quad F\in\mathcal{C}^1(\mathbb{R}) \,,
\end{equation}
or, reverting back to the original states $u$ and $v$,
\begin{equation}
\label{eqn:nonlinear_model_rotation_X}
    \hat{Y}_R=F\left(\ln\left(\sqrt{\frac{u^2+v^2}{|1-\left(u^2+v^2\right)|}}\right)-\tan^{-1}\left(\frac{v}{u}\right)\right)\partial_t-v\partial_u+u\partial_v \,, \quad F\in\mathcal{C}^1(\mathbb{R})\,.
\end{equation}

As in the previous example, the transformation $\Gamma_3^R$ in the time domain generated by $\hat{Y}_R$ is given by the exponential map $\exp \mathrm{(} \varepsilon \hat{Y}_R \mathrm{)}$. The action on solutions to the system in Eq.~\eqref{eqn:nonlinear_model_ODE} of $\Gamma_3^R$ is illustrated in Fig.~\ref{fig:nonlinear_model_symmetries_ODE} for different choices of the arbitrary function $F$ in Eq.~\eqref{eqn:nonlinear_model_rotation_X}.

\begin{figure}[ht!]
\begin{center}
\includegraphics[width=\textwidth]{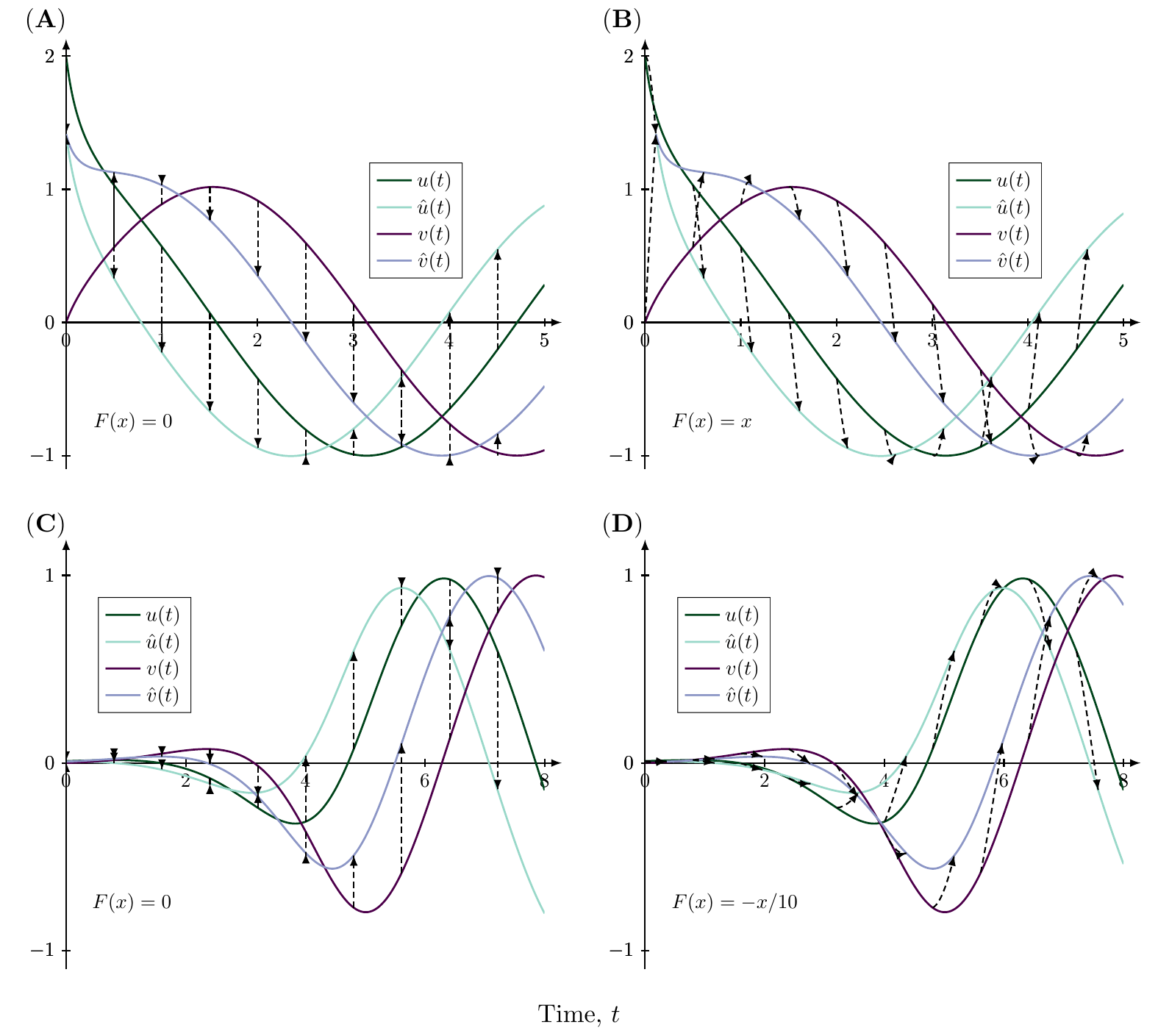}
\caption{\textit{Action of lifted phase plane symmetries for the non-linear oscillator model .} Original solution curves $\left(u(t),v(t)\right)$ and transformed solution curves $\left( \hat{u}(t),\hat{v}(t) \right) = \Gamma_3^R\left(u(t),v(t)\right)$ generated by $\hat{Y}_R$ in Eq.~\eqref{eqn:nonlinear_model_rotation_X} for (\textbf{A}) $F(x)=0$ and $r>1$, (\textbf{B}) $F(x)=x$ and $r>1$, (\textbf{C}) $F(x)=0$ and $r<1$, (\textbf{D}) $F(x)=-x/10$ and $r<1$. Dashed arrows represent the transformation $\Gamma_3^R$.}
\label{fig:nonlinear_model_symmetries_ODE}
\end{center}
\end{figure}

We end this example by considering the family of solutions to the system in Eq.~\eqref{eqn:nonlinear_model_ODE} obtained by continuously varying the transformation parameter $\varepsilon$ for $\Gamma_3^R$. Any 1-parameter group of Lie symmetries can be expressed as a translation in some appropriate coordinate system of canonical coordinates. In the case of Eq.~\eqref{eqn:nonlinear_model_phase} the canonical coordinates in the phase plane are the polar coordinates $(\theta,r)$ considered above, where $\hat{Y}_R$ acts like a translation in the angular coordinate while leaving the radial coordinate invariant as illustrated in Fig.~\ref{fig:nonlinear_model_symmetries_phase}. In the time domain, however, the change in state space is also accompanied by a non-trivial transformation in the time direction, resulting in the qualitative behaviour illustrated in Fig.~\ref{fig:nonlinear_model_solutions_family}.

\begin{figure}[ht!]
\begin{center}
\includegraphics[width=\textwidth]{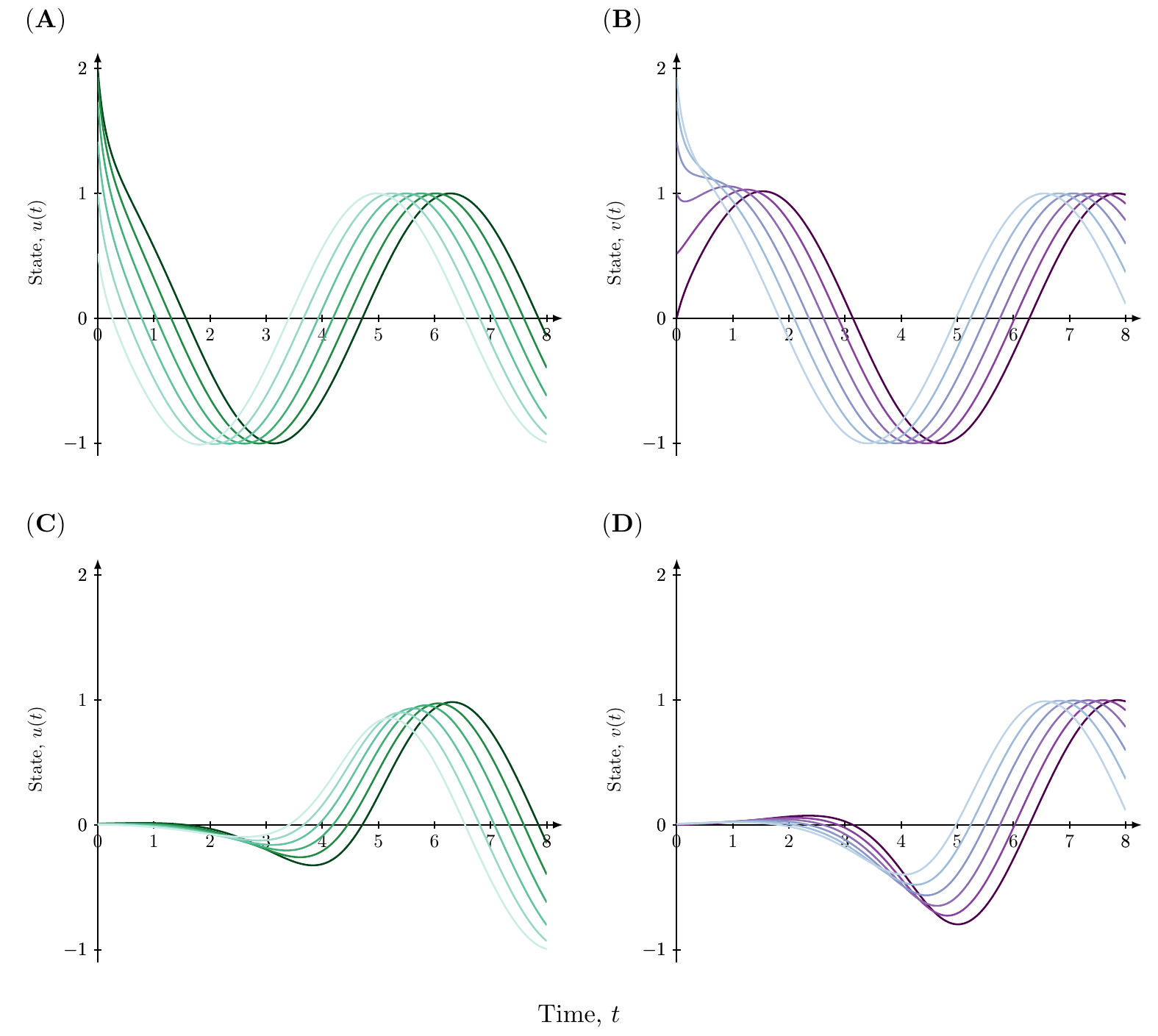}
\caption{\textit{Family of lifted solutions to the non-linear model}. The family of solutions is generated by repeatedly transforming solutions with $F(x)=0$ in $\xi_R$ given in Eq.~\eqref{eqn:xi_rotation_polar}. Transformations of solutions are illustrated for $r>1$ in (\textbf{A}) $u(t)$ and (\textbf{B}) $v(t)$, and for $r<1$ in (\textbf{C}) $u(t)$ and (\textbf{D}) $v(t)$.}
\label{fig:nonlinear_model_solutions_family}
\end{center}
\end{figure}
\section{Discussion}
In this paper, we show that for autonomous two-state dynamical systems, symmetries in the time domain induce symmetries in the phase plane through the reduction given in Eq.~\eqref{eqn:reduction_map}. Conversely, we show that it is possible to lift phase plane symmetries to symmetries in the time domain, which is less obvious because it involves showing that there is no obstruction to simultaneously satisfying both symmetry conditions in Eq.~\eqref{eqn:sym_con_ODE} in the time domain. The lift amounts to solving the \textit{lifting condition} in Eq.~\eqref{eqn:lifting_condition} which is a linear PDE for the tangent $\xi(t,u,v)$ in the time direction, and consequently the method of characteristics is directly applicable. We also show that symmetries in the phase plane, and their connections to symmetries in the time domain, are independent of the arbitrary choice of independent variable in the reduction to the $(u,v)$-plane. Our results constitute a novel and important contribution to the understanding of the relationship between symmetries in two complementary descriptions of dynamical systems, and in particular entail that any symmetry-based analysis in the phase plane can be understood in the time domain through the lifting condition. 

A symmetry of a phase plane ODE that is common to different time-dependent ODEs can be lifted to distinct symmetries in the time domain, provided the additional information regarding the time domain dynamics given by $\omega_u(u,v)$ and $\omega_v(u,v)$. This is most clearly illustrated when studying the mass-conserved linear model in Eq.~\eqref{eqn:linear_model_ODE}. The phase plane ODE of any autonomous mass-conserved system in the time domain, given by $\dot{u}=-\dot{v} =\omega(u,v)$ for an arbitrary function $\omega$, is $\dd v/\dd u=-1$ as in Eq.~\eqref{eqn:linear_model_phase}. While this phase plane ODE is common to all mass conserved systems, the lifting condition in Eq.~\eqref{eqn:lifting_condition} allows us to lift the two phase plane symmetries generated by $Y_S$ in Eq.~\eqref{eqn:linear_model_scaling_Y} and $Y_G$ in Eq.~\eqref{eqn:linear_model_genrot_Y} to generators $\hat{Y}_S$ and $\hat{Y}_R$, respectively, of symmetries in the time domain, since it accounts for the choice of the function $\omega(u,v)$.

It deserves to be emphasised that autonomy of the system can itself be described in terms of symmetries. The fact that Eq.~\eqref{eqn:system_ODE} contains no explicit time dependence amounts to symmetry under time translations generated by $X = \partial_t$. The induced vector field is the trivial one $f_*X = 0$, which leaves $M_2$ and therefore $J_3$ invariant. Conversely, the generator of time translations is always a solution to the lifting condition $\left. \left( D_t \xi \right) \right|_{\Delta} = 0$ for the identity transformation in phase space.

In addition, we noted above that the constant of integration appearing in the solution to the characteristic system in Eq.~\eqref{eqn:characteristic_system_lifting} corresponding to the lifting condition can be interpreted as generating an arbitrary constant time translation following the lift. This implies that all the non-trivial information required to lift the phase plane symmetry generator $Y$ to a generator of symmetries in the time domain is contained in the non-homogeneous term in Eq.~\eqref{eqn:characteristic_system_solution}, a fact which is also illustrated in Fig.~\ref{fig:linear_model_symmetries_ODE} and Fig.~\ref{fig:nonlinear_model_symmetries_ODE}.

While our treatment in this paper is limited to two-state systems, it would be very interesting to explore a generalisation to higher dimensions. The dynamics of autonomous systems in more than two states is also frequently analysed in phase space, and the same motivation we have offered above for understanding the connection between symmetries in different formulations applies. Generalising our analysis in a different direction, it is natural to consider the case of higher-dimensional symmetry groups, i.e., generated by a set of more than one vector field. The generators form a Lie algebra whose structure in terms of commutation relations determine the properties of the group of transformations. Exploring how such algebraic properties in the time domain and in the phase plane are related would provide insight into the correspondence between model structures in the two representations.  
\section{Acknowledgements}
FO would like to thank the Wolfson Centre for Mathematical Biology for hospitality and the Kempe Foundation for financial support during the conception of this work. JGB would like to thank the Wenner--Gren Foundation for a Research Fellowship and Linacre College, Oxford, for a Junior Research Fellowship. 

\section{Author contributions}
All three authors conceptualised the work, analysed the results and wrote the paper. FO formulated and proved the three theorems. JGB designed the examples and made the figures. 

\clearpage
\bibliographystyle{unsrt}
\bibliography{symmetry_bib}

\end{document}